\documentclass[a4paper,12pt]{article}
\usepackage{amssymb,amsmath}
\usepackage[dvips]{color}
\usepackage[dvips]{graphics}
 \usepackage{graphicx}

\newcommand{\eps}{{\epsilon}_{M}}
\newcommand{\rest}{{\cal{O}}({\eps}^2)}

\newtheorem{theorem}{Theorem}

\newtheorem{definition}{Definition}

\newtheorem{corollary}{Corollary}[section]

\begin{document}
\title{ \Large\bf  On the accuracy and stability  of algorithms most commonly used in the evaluation of Chebyshev polynomials of the first kind}
\author{Alicja Smoktunowicz  \thanks{Faculty of Mathematics and Information Science, Warsaw
University of Technology, Koszykowa 75, Warsaw, 00-662 Poland, e-mail: smok@mini.pw.edu.pl} \and
Agata Smoktunowicz \thanks{School of Mathematics, University of Edinburgh, Edinburgh, Scotland EH9 3JZ, UK, e-mail: A.Smoktunowicz@ed.ac.uk.
The research  of Agata Smoktunowicz was funded by ERC grant 320974.
} \and
Ewa Pawelec  \thanks{Faculty of Mathematics and Information Science, Warsaw
University of Technology, Koszykowa 75, Warsaw, 00-662 Poland, e-mail: E.Pawelec@mini.pw.edu.pl}
}

\maketitle

\begin{abstract}
This paper provides error analyses of the algorithms most commonly used for the evaluation of
the  Chebyshev polynomial of the first kind $T_N(x)$. Some of these algorithms are shown to be backward  stable.
This means that the computed value of $T_N(x)$ in floating point arithmetic
by these algorithms can be interpreted as a slightly perturbed value of polynomial
$T_N$, for slightly perturbed value of $x$.
\end{abstract}

\medskip

\noindent


\noindent


\section{Introduction}
\label{introd1}
Chebyshev polynomials of the first kind  $(T_n(x))$ are widely used in many applications. They satisfy the three-term recurrence
\begin{equation}\label{eqs0}
T_{n}(x)= 2 x T_{n-1}(x)- T_{n-2}(x), \, n=2,3, \ldots,
\end{equation}
where  $T_{0}(x)= 1,\, T_{1}(x) = x$.

There are several algorithms for evaluating $T_N(x)$ (see \cite{barrio02},  \cite{barrio03},\cite{koepf99}, \cite{smok02}).
However, for numerical purposes some of them are poor (see \cite{Bakhvalov1971}, \cite{gen69}, \cite{koepf99}).
For example, using the symbolic calculations in \textsl{MATHEMATICA}, \textsl{MAPLE}, \textsl{DERIVE} and others packages, it is possible to find
the  expanded  form of $T_N(x)$, that is,  the exact coefficients $a_n$ of $T_N(x)$ such that $T_N(x)=a_0+a_1 x + \cdots + a_N  x^N$.
However  computing the value $T_N(x)$ at a given floating point $x$ from this form can be disastrous.
At first this may seem surprising, since  the coefficients $a_n$ are integers. Note that there are large $a_n$ for large $N$, for example
 the leading coefficient $a_N=2^{N-1}$.

Symbolic and numeric computations often demand different approaches (see \cite{koepf99}).
In practice, a desirable property for an algorithm is numerical stability (see \cite{wilk63}). Our problem of computing the value $T_N(x)$ at a given point $x$  is a special case of  the  general problem of evaluating the polynomial $p_N(x)=c_0 T_0(x)+ c_1 T_1(x)+ \ldots + c_N T_N(x)$. Clenshaw's and Forsythe's algorithms are recommended here. An error analysis of Clenshaw's algorithm in the general case was first provided by
D. Elliott in \cite{ell68}. See also \cite{deu76}, \cite{gen69}, \cite{smok02}, \cite{barrio02}--\cite{barrio03}, where the authors gave
 the forward error bounds  for the evaluation of  $p_N(x)$ in floating point arithmetic.
However, it is of interest to know whether an algorithm  is backward stable with respect to the data $x$.
Roughly speaking, the computed value $\tilde T_N(x)$ by a backward  stable
algorithm can be interpreted as a slightly perturbed value
of the polynomial $T_N$ for a slightly perturbed value of $x$.
A more precise definition is now given.
\begin{definition}\label{def1}  An algorithm $W$ of computing $T_N(x)$
is backward stable with respect to the data $x$ if the value
$\tilde T_N(x)$ computed by $W$ in floating point arithmetic satisfies
\begin{equation}\label{eqs1}
\tilde T_{N}(x)=(1+\delta_N) T_N((1+\Delta_N) x)+ \rest,\,\,
\vert \delta_N \vert, \vert \Delta_N \vert \leq \eps L,
\end{equation}
where $L=L(N)$ is a modest constant and $\eps$ is machine precision.
\end{definition}

Throughout this paper we will ignore the terms of order $\rest$. It is easy to check that (\ref{eqs1})  is equivalent to
\begin{equation}\label{eqs2}
|\tilde T_{N}(x)-T_{N}(x)| \leq   \eps \, L \, C_N(x) + \rest,
\end{equation}
where
\begin{equation}\label{err}
C_N(x) = |T_N(x)| + |x T_N'(x)|.
\end{equation}
Note that
\begin{equation}
C_N(x) = |T_N(x)| + N |x U_{N-1}(x)|,
\end{equation}
where $U_{N-1}(x)$ denotes the Chebyshev polynomial of the second kind.
These polynomials satisfy the recurrence relations
\begin{equation}\label{eqs00}
U_{n}(x)= 2 x U_{n-1}(x)- U_{n-2}(x), \, n=2,3, \ldots,
\end{equation}
where $U_{0}(x)= 1,\, U_{1}(x) = 2 x$.

We will consider the following algorithms for computing $T_N(x)$ at a given point $x \in [-1, 1]$.

\begin{itemize}
\item {\bf Algorithm I ($\bf  Three-term \, recursion $)} \\
$T_0=1; \quad T_1=x;$\\
$T_{n}= 2 x T_{n-1}- T_{n-2}$ for $n=2,3, \ldots, N.$\\
$T_{N}(x)=T_{N}.$

\item {\bf Algorithm II ($\bf Fast$)} \\
Let $N=2^p$.\\
This algorithm uses the identity $T_{2n}(x)= T_2(T_n(x))$\\
and computes $R_n=T_{2^n}(x)$ as follows:\\
$R_0=x;$\\
$R_{n}= 2\, {R_{n-1}^2}- 1$ for $n=1,\ldots, p.$\\
$T_{N}(x)=R_{p}.$

\item {\bf Algorithm III ($\bf Trigonometric$)} \\
$T_N(x)=\cos(N* \arccos(x)).$

\item  {\bf Algorithm IV ($\bf Horner$)} \\
Use Horner's scheme for the  expanded form of $T_N(x)$: \\
 $T_N(x)= 2^{N-1} x^N + a_{N-1} x^{N-1} + \ldots + a_0.$\\
 Note that the coefficients $a_n$ are integers.
\end{itemize}

\medskip

The rest of this paper is organized as follows. In Section 2 we recall some basic properties of the Chebyshev polynomials. In Section 3 we will use these properties in a derivation of the lower and upper bounds for $C_n(x)$.
In Section 4 we present the error  analyses for Algorithms I and II above, proving that these algorithms
are backward stable in the sense of  (\ref{eqs2}).
In Section 5 we compare the accuracy of the algorithms using  numerical experiments performed in \textsl{MATLAB};
our tests show  that Algorithm III  can be less accurate for $x$ near $\pm 1$ and that
Algorithm IV is not always backward stable.

\medskip

\section{Preliminaries}
We will need some properties of the Chebyshev polynomials (see \cite{pasz75} and \cite{szego59}). For  $-1 \leq x \leq  1$ we have   $T_n(x)=\cos (n \Theta)$, where $\Theta=\arccos x$ and
$U_{n-1}(x)= \sin (n \Theta)/{\sin \Theta}$ for $0<x<1$.

The following identities hold
\[
U_{n-1}(x)=\frac{{T_{n}}'(x)}{n},
\]
\[
T_n(-x)=(-1)^n T_n(x), \quad U_n(-x)=(-1)^n U_n(x).
\]

The Chebyshev polynomials of the first kind satisfy the following differential equations
\begin{equation}\label{wzor003}
(1-x^2) {T_n''(x)} -x {T_n'}(x) + n^2 T_n(x)=0
\end{equation}
and
\begin{equation}\label{wzor3}
{T_n^2(x)} +  \frac{1-x^2}{n^2} \, {T_n'}^2(x) =1.
\end{equation}
The last equality is a consequence of the trigonometric identity $\cos^2 n \theta + \sin^2 n \theta =1$.

For  $-1 \leq x \leq  1$ and $n=0,1, \ldots$ we have the upper bounds
\begin{equation}\label{ineq2}
|T_n(x)| \leq |T_n(1)|= 1, \quad |U_n(x)| \leq |U_n(1)|=n+1
\end{equation}
and  for  $-1 < x < 1$
\begin{equation}\label{ineq3}
|U_n(x)| \leq  \frac{1}{\sqrt{1-x^2}}.
\end{equation}

The roots $(t_i)$ of $T_n(x)$  are distinct and belong to $(-1,1)$:
\begin{equation}\label{zeraT}
t_i=\cos \frac{ (2 i -1) \pi}{2 n}, \quad  i=1, 2, \ldots, n.
\end{equation}

The roots $(u_i)$ of $T_n'(x)$ (i.e. the roots of $U_{n-1}(x)$)  are:
\begin{equation}\label{zeraU}
u_i=\cos \frac{ i  \pi}{n}, \quad  i=1, 2, \ldots, n-1.
\end{equation}

Then $-1< t_n < u_{n-1} < \ldots < u_1 < t_1  < 1$ and
\begin{equation}\label{nowe11}
T_n(u_i)=(-1)^i \quad i=1, 2, \ldots, n-1.
\end{equation}


For  $-1 \leq x \leq 1$ and $m=0,1, \dots $ we get
\begin{equation}\label{wzorek3}
|T_{2m+1}(x)| \leq (2m+1) |x|,
\quad
|U_{2m+1}(x)| \leq 2(m+1) |x|.
\end{equation}

In evaluating  the Chebyshev polynomials one can use the composition identity
\begin{equation}\label{wzor0}
T_{mn}(x)=T_m(T_n(x)), \quad  m, n=0,1, \ldots.
\end{equation}

\medskip

\section{Lower and upper  bounds for  $C_{n}(x)$}
Since $C_n(-x)=C_n(x)$ for all $x$, we restrict our considerations to the interval $[0,1]$.
From (\ref{ineq2}) it follows that   $C_n(x) \leq C_n(1)= n^2+1$ for $0 \leq x \leq 1$.
By (\ref{zeraU})--(\ref{nowe11})  we have $C_n(u_i)= 1$ for $i=1, \ldots, n-1$.
If $n$ is odd then  $C_n(0)=0$.

%

\begin{theorem}\label{tw1}
Let $n$ be a natural number. Assume that $s_n \leq x \leq 1$,
where
\begin{equation}\label{eqs11}
s_n=\frac{1}{\sqrt{n^2+1}}.
\end{equation}

Then we have
\begin{equation}\label{eqs12}
C_n(x)=|T_n(x)| + |x T_n'(x)| \ge 1.
\end{equation}
\end{theorem}

\medskip

\begin{proof}
Notice that the inequality $x^2 \ge s_n^2$ is equivalent to $x^2 \ge \frac{1-x^2}{n^2}$.
From this and (\ref{wzor3}) we get
\[
{C_n^2(x)} \ge {T_n^2(x)} + x^2 {T_n'^2(x)} \ge {T_n^2(x)} +  \frac{1-x^2}{n^2} \, {T_n'}^2(x) =1.
\]
The proof is now complete.
\end{proof}

\begin{theorem}\label{tw2}
Let  $n$ be a natural number.  Assume that $0 \leq x \leq s_n$,
where $s_n$ is defined by (\ref{eqs11}).
Then
\begin{description}
\item[(i)] $C_{n}(x) \ge n |x|$ for all $n$,
\item[(ii)]$C_{n}(x) \ge 1$ for even $n$.
\end{description}
\end{theorem}

\medskip

\begin{proof} We consider case {\bf (i)}.
Clearly, $1 \ge n^2\, x^2$, by (\ref{eqs11}) and since $0 \leq x \leq s_n$. Therefore,
\[
{C_n^2(x)} \ge 1 {T_n^2(x)} + x^2 {T_n'^2(x)} \ge {n^2 x^2 T_n^2(x)} + x^2 {T_n'^2(x)} \ge x^2 n^2 (T_n^2(x) +  \frac{1}{n^2} \, {T_n'}^2(x)).
\]
Since $1 \ge 1-x^2$ we get
\[
{C_n^2(x)} \ge  x^2 n^2 (T_n^2(x) +  \frac{1-x^2}{n^2} \, {T_n'}^2(x))= x^2 n^2,
\]
due to   (\ref{wzor3}). Therefore,  $C_n(x) \ge n |x|$. This completes the proof of case {\bf (i)}.

\medskip

Now we consider case {\bf (ii)}. Let $n=2m$. We first prove that $T_{2m}$ has no roots in $(0,s_{2m})$. By (\ref{zeraT}), we need to show that
\begin{equation}\label{eqs24}
t_m= \cos \frac{ (2 m -1) \pi}{4m} > s_{2m}.
\end{equation}
Notice that
\[
t_m= \cos (\frac{\pi}{2} -  \frac{\pi}{4m})= \sin \frac{\pi}{4m}.
\]

Since $0< \tan \Theta > \Theta$ for all $0 < \Theta < \frac{\pi}{2}$, we have $\tan^2 {\Theta} > {\Theta}^2$.
From this it follows that  $\sin^2 \Theta > \frac{\Theta^2}{1+ \Theta^2}$.
Substituting $\Theta= \pi/{4m}$ in the above inequality leads to
\[
t^2_m > \frac{\pi^2}{16 m^2 + \pi^2} > \frac{1}{4m^2 +1}= s^2_{2m},
\]
so $t_m > s_{2m}$. This finishes the  proof of (\ref{eqs24}).

We see that $T_{2m}$ has no roots in  $(0,s_{2m})$.  Moreover, $T_{2m}(0)=(-1)^m$ and $T_{2m}'(0)=0$.
We conclude from (\ref{zeraT})--(\ref{zeraU}) that $0$ is the only root of $T_{2m}'$ in the interval $(-s_{2m}, s_{2m})$.
Notice that $T_{2m}$ and $T_{2m}''$ are even, i.e. $T_{2m}(-x)=T_{2m}(x)$
and $T_{2m}''(-x)=  T_{2m}''(x)$ for all $x$. $T_{2m}'$ is odd, that is, $T_{2m}'(-x)=  - T_{2m}'(x)$.
Thus we see  that the polynomials $T_{2m}$ and  $T_{2m}'$ do not change the signs in $(0,s_{2m})$.

More precisely, if $m$ is even, then for all $0 < x < s_{2m}$ we have
$T_{2m}(x)>0$ and $T_{2m}'(x) < 0$, hence $C_{2m}(x)=T_{2m}(x) -x T_{2m}'(x)$.
Similarly, if $m$ is odd then  $T_{2m}(x)<0$ and $T_{2m}'(x)>0$, so  $C_{2m}(x)= -T_{2m}(x) +x T_{2m}'(x)$.
We see that $C_{2m}'(x) = -  T_{2m}''(x)$ if $m$ is even and  $C_{2m}'(x) =  T_{2m}''(x)$ otherwise.

By (\ref{wzor003}) for $n=2m$, we obtain  the formula
\[
(1-x^2) {T_{2m}''(x)} = x {T_{2m}'}(x) - {2m}^2 T_{2m}(x).
\]
We see that for all $0 < x < s_{2m}$ we have $T_{2m}''(x) < 0 $ if $m$ is even and $T_{2m}''(x) > 0 $ if $m$ is odd.
We conclude that $C_{2m}'(x) > 0$ for any $m$, so  $C_{2m}(x)$ is increasing in the interval $(0,s_{2m})$. This gives
the lower bound $C_{2m}(x) \ge C_{2m}(0)=1$. The proof of our theorem  is now complete.
\end{proof}
\medskip

\section{Error analysis}
As a direct  consequence of Theorems \ref{tw1}--\ref{tw2} we obtained the following result.

\begin{corollary}\label{corr1}
Let $N \ge 2$ and $s_N=\frac{1}{\sqrt{N^2+1}}$.
Assume that an algorithm $W$ evaluates $T_N(x)$ in floating point arithmetic with the small forward  error
\begin{equation}\label{blad1}
|\tilde T_{N}(x)-T_{N}(x)| \leq   \eps \, L_1  + \rest,
\end{equation}
where  $L_1=L_1(N)$ is a modest constant and $\eps$ is machine precision. Then
\begin{description}
\item[(i)] if $N$ is even then $W$ is backward stable in $[-1,1]$, i.e. (\ref{eqs2}) holds  with the constant $L=L_1$,
\item[(ii)] if $N$ is odd then $W$ is backward stable for $s_N \leq |x| \leq 1$ with the constant $L=L_1$,
\item[(iii)] if $N$ is odd and there is a small constant $L_2=L_2(N)$ such that for $|x| \leq s_N$ we have
\begin{equation}\label{blad2}
|\tilde T_{N}(x)-T_{N}(x)| \leq   \eps \, L_2 |x|  + \rest,
\end{equation}
then $W$ is backward stable for $|x| \leq s_N$ with the constant $L=L_2/N$.
\end{description}
\end{corollary}

\subsection{Error analysis of Algorithm I}
We analyze  the rounding errors in Algorithm I.
\begin{theorem} Let $N \ge 2$ and $s_N=\frac{1}{\sqrt{N^2+1}}$.
Let $\tilde T_{n}$ denote the quantities computed
by Algorithm I  in floating point arithmetic fl with  machine precision $\eps$.
Let $\tilde T_N(x)= \tilde T_{N}$.
Assume that $x$ is exactly representable in fl ($fl(x)=x$) and $x \in [-1,1]$.

Then  we have the bound
\begin{equation}\label{w1}
|\tilde T_N(x)  - T_N(x)| \leq \eps \, \frac{3 N (N-1)}{2} + \rest.
\end{equation}

If $|x| \leq s_N$ then
\begin{equation}\label{w2}
|\tilde T_N(x)  - T_N(x)| \leq \eps \, \frac{9 (N-1)}{2} + \rest.
\end{equation}

Moreover, if $|x| \leq s_N$ and $N$ is odd  then
\begin{equation}\label{w3}
|\tilde T_N(x)  - T_N(x)| \leq \eps \frac{5 (N-1)(N+7)}{8} \, |x| + \rest.
\end{equation}
\end{theorem}

\begin{proof}
Note that $\tilde T_0=1$, $\tilde T_1 =x$ and for $n=2, \ldots$ we have
\[
\tilde T_{n}= (2 x \, \tilde T_{n-1} (1+\alpha_n)- \tilde T_{n-2})(1+\beta_n),\quad
 |\alpha_{n}|, |\beta_n| \leq \eps.
 \]
We rewrite it as follows
\begin{equation}\label{nowy1}
\tilde T_{n}= 2 x \, \tilde T_{n-1} - \tilde T_{n-2}  + \xi_n,\quad
 \xi_n=2 x  \tilde T_{n-1}  \alpha_n + \frac{\beta_n}{1+\beta_n} \tilde T_{n}.
 \end{equation}

Let $e_n= \tilde T_{n} - T_n(x)$. We observe that  $e_0=e_1=0$ and $e_n= 2 x e_{n-1}-e_{n-2} + \xi_n$ for $n=2, 3, \ldots, N$.
From this  it follows that
\[
e_N= \tilde T_{N}- T_N(x)= \sum_{n=2}^{N} {U_{N-n}(x) \xi_n}.
\]

Therefore,
\[
|e_N| \leq \sum_{n=2}^{N} {|U_{N-n}(x)| \, |\xi_n}|.
\]
This together with (\ref{nowy1}) leads to
\begin{equation}\label{nowy2}
|\xi_n| \leq \eps \, (2 |x|\, |T_{n-1}(x)| + |T_{n}(x)|) +  \rest,
\end{equation}
hence
\begin{equation}\label{nowy3}
|e_N| \leq \eps \,  \sum_{n=2}^{N}{(2 |x|\, |T_{n-1}(x)| + |T_{n}(x)|)\, |U_{N-n}(x)|} + \rest.
\end{equation}

Since $|T_n(x)| \leq 1$  for $|x| \leq 1$  we obtain
\begin{equation}\label{nowy4}
|e_N| \leq \eps \, 3 \,  \sum_{n=2}^{N}{|U_{N-n}(x)|} + \rest.
\end{equation}

This together with (\ref{ineq2}) leads to
\[
|e_N| \leq \eps \,  3 \, \sum_{n=2}^{N}{(N-n+1)} + \rest \leq \eps \, \frac{3 N (N-1)}{2} + \rest.
\]
The proof of (\ref{w1}) is complete.

\medskip

Now consider the case  $|x| \leq s_N$. By (\ref{ineq3})  we get
$|U_k(x)| \leq \frac{1}{\sqrt{1-s_N^2}}$ for  $k=0,1, \ldots$.

Therefore,
\begin{equation}\label{nowy5}
|U_k(x)| \leq  \frac{3}{2} \, \mbox{ for } |x| \leq s_N, \quad  k=0,1, \ldots.
\end{equation}
From this and  (\ref{nowy4})  the bound (\ref{w2}) follows immediately.

\medskip

Now assume that $N$ is odd and $|x| \leq s_N$. We rewrite (\ref{nowy3}) as follows
\begin{equation}\label{nowy55}
|e_N| \leq \eps \, (A_N(x) + B_N(x)) + \rest,
\end{equation}
where
\begin{equation}\label{nowy6}
A_N(x)= 2 |x| \sum_{n=2}^{N}{|T_{n-1}(x)|\, |U_{N-n}(x)|},
\end{equation}
\begin{equation}\label{nowy7}
B_N(x)= \sum_{n=2}^{N}{|T_{n}(x)|\, |U_{N-n}(x)|}.
\end{equation}

This together with (\ref{nowy5}) and the inequality  $|T_{n-1}(x)| \leq 1$   gives
\begin{equation}\label{AN}
A_N(x) \leq 3 |x| (N-1).
\end{equation}

To estimate $B_N(x)$ for $N=2m+1$  we split it as follows
\[
B_N(x)= \sum_{k=1}^{m}{|T_{2k}(x)|\, |U_{N-2k}(x)|} + \sum_{k=1}^{m}{|T_{2k+1}(x)|\, |U_{N-(2k+1)}(x)|}.
\]

Note that (\ref{wzorek3}) implies the following upper bounds (for the polynomials of the odd degrees)
\[
|U_{N-2k}(x)| \leq (N-2k+1) \, |x|, \quad |T_{2k+1}(x)| \leq (2k+1)\, |x|.
\]

By (\ref{nowy5}),  we have $|U_{N-(2k+1)}(x)| \leq  \frac{3}{2}$  for $|x| \leq s_N$.
We conclude that
\[
B_N(x) \leq  \left(\sum_{k=1}^{m}{1 \, (N-2k+1)\, |x|} + \frac{3}{2}\, \sum_{k=1}^{m}{(2k+1)\, |x|}\right).
\]

The last inequality together with (\ref{nowy55}) and (\ref{AN}) leads to
\[
|e_N| \leq \eps \, (3 (N-1) + m(N-m)+ \frac{3}{2}\, m (m+2) )\, |x| + \rest.
\]

Since $m=(N-1)/2$ we get immediately (\ref{w3}).
\end{proof}

\medskip

By Corollary \ref{corr1} we conclude that Algorithm I is backward stable in $[-1,1]$ with the constant $L$ of order $N^2$.
Algorithm I is backward stable with the constant $L$ of order $N$ for   $|x| \leq s_N$.

\subsection{Error analysis of Algorithm II}

\begin{theorem} Let $N=2^p$ and  $\tilde R_{n}$ denote the quantities computed
by Algorithm II  in floating point arithmetic fl with machine precision $\eps$. Let $\tilde T_N(x)= \tilde R_{p}$.
Assume that $fl(x)=x$ and $x \in [-1,1]$.

Then
\begin{equation}\label{wzor100}
|\tilde T_N(x) - T_N(x)| \leq \eps N^2 + \rest
\end{equation}
and (\ref{eqs2}) holds  with the constant $L=N^2$.
\end{theorem}

\begin{proof}
We  see that $\tilde R_0=x$ and for $n=1,2, \ldots, p$ we have
\[
\tilde R_{n}= (2 \, {\tilde R_{n-1}^2} (1+\alpha_n)- 1)(1+\beta_n),\quad
 |\alpha_{n}|, |\beta_n| \leq \eps.
 \]
From this it follows that
\begin{equation}\label{fast1}
\tilde R_{n}= 2 \,\, {\tilde R_{n-1}^2} - 1 + \xi_n,\quad
 \xi_n=2 \,\, {\tilde R_{n-1}^2}\alpha_n + \frac{\beta_n}{1+\beta_n} \tilde R_{n}.
 \end{equation}
 We can prove by induction on $n$ that
 \[
\tilde R_n - R_n = \sum_{k=1}^{n-1}{4^{n-k} T_{2^k}(x) T_{2^{k+1}}(x) \cdots T_{2^{n-1}}(x) \, \xi_k}+ \, \xi_n+ \rest.
 \]
Since  $|T_k(x)| \leq 1 $ for $x \in [-1,1]$ we obtain
\[
|\tilde R_n - R_n| \leq \sum_{k=1}^{n}{4^{n-k}\, |\xi_k}| + \rest.
 \]
This together with (\ref{fast1}) gives $|\xi_k| \leq 3 \eps + \rest$, so
\[
|\tilde R_n - R_n| \leq \eps \, 3 \, \sum_{k=1}^{n}{4^{n-k}} + \rest.
\]
Finally, for $n=p$ we get  the following upper bound on $\tilde T_N(x)= \tilde R_{p}$
\begin{equation}
\vert \tilde T_N(x) - T_N(x) \vert \leq \eps N^2 + \rest.
\end{equation}

From Corollary \ref{corr1} we conclude that (\ref{eqs2}) holds  with the constant $L=N^2$, so
Algorithm II is backward stable.
\end{proof}

\section{Numerical tests}
To illustrate our results we present numerical tests  in \textsl{MATLAB} with machine precision  $\eps= 2^{-52} \approx  2.2 \cdot 10^{-16}$.
We compare the results computed by  Algorithms I--IV with the  exact values of the  Chebyshev polynomial $T_N(x)$.
They  were obtained by implementing Algorithm I in high precision using the VPA (Variable Precision Arithmetic) function from \textsl{MATLAB}'s
Symbolic Math Toolbox and then rounded to $16$th decimal digits.
We compute the relative error
\begin{equation}\label{err1}
e_N = \frac{\max_{ x \in S} {|T_N(x)- \tilde T_N(x)|}}{\eps}.
\end{equation}

Here  $S$ consists of  $p$th equally spaced checkpoints $t_1,t_2, \ldots, t_p$ from the interval $[a,b]$, where $-1 \leq a < b \leq 1$,
i.e. $t_i=a+(i-1)  h$, $i=1, 2,  \ldots, p$ and $h=(b-a)/(p-1)$.

\begin{center} \noindent {\footnotesize Table 1: The error (\ref{err1}) for Algorithms I--IV in $[-1,1]$
and $h=1/100$}.\\

\bigskip

\nopagebreak

\begin{tabular}{|c|c|c|c|c|}
\hline $N$ & Algorithm I & Algorithm II & Algorithm III & Algorithm IV \\
\hline $8$ & $5.25$ & $6.68$ & $13.12$  & $95.68$  \\
\hline $16$ & $11.00$ & $12.00$ & $22.37$ & $3.48e\!+\!04$  \\
\hline $32$ & $21.78$ & $43.00$ & $55.12$ & $3.13e\!+\!10$  \\
\hline $64$ & $35.00$ & $98.75$ & $88.50$ & $4.83e\!+\!22$  \\
\hline $128$ & $66.00$ & $257.00$ & $193.50$  & $2.88e\!+\!47$  \\
\hline $256$ & $165.00$ & $888.75$ & $410.25$  & $1.09e\!+\!96$  \\
\hline $512$ & $280.75$ & $1770.0$ & $841.87$   & $1.61e\!+\!194$ \\
\hline $1024$ & $679.62$ & $3570.0$ & $1783.20$  & NaN \\
\hline
\end{tabular}
\end{center}

\medskip

We see that Algorithm IV is poor as a method of evaluating the Chebyshev polynomial $T_N(x)$, even for $N \ge 16$.
The best results are produced by Algorithm I. These tests indicate that Algorithm II is less accurate than Algorithm I.

\medskip

\begin{center} \noindent {\footnotesize Table 2: The error (\ref{err1}) for Algorithms I and III in  $[-0.8,-0.6]$
and $h=1/1000$. }\\

\bigskip

\nopagebreak

\begin{tabular}{|c|c|c|}
\hline $N$ & Algorithm I & Algorithm III  \\
\hline $100$ & $35.500$ & $176.125$  \\
\hline $300$ & $104.125$ & $607.750$   \\
\hline $500$ & $164.50$ & $1008.0$  \\
\hline $800$ & $262.25$ & $1355.0$ \\
\hline $900$ & $289.50$ & $2159.0$  \\
\hline $1000$ & $340.34$ & $2137.0$  \\
\hline
\end{tabular}
\end{center}

\medskip

\begin{center} \noindent {\footnotesize Table 3: The error (\ref{err1}) for Algorithms I and III in  $[-1,-0.8]$
and $h=1/1000$. }\\

\bigskip

\nopagebreak

\begin{tabular}{|c|c|c|}
\hline $N$ & Algorithm I & Algorithm III  \\
\hline $101$ & $73.62$ & $267.87$  \\
\hline $301$ & $212.37$ & $561.50$   \\
\hline $501$ & $356.62$ & $1144.8$  \\
\hline $801$ & $549.09$ & $1710.3$ \\
\hline $901$ & $665.06$ & $1841.0$  \\
\hline $1001$ & $672.53$ & $2453.4$  \\
\hline
\end{tabular}
\end{center}

\medskip

These tests show  that Algorithm III  can be much less accurate than Algorithm I for $x$ near $-1$.
Numerical properties of Algorithm III strongly depend upon the accuracy of computing the trigonometric functions {\bf cos} and {\bf arcos}.
For a deeper discussion of the accuracy of the evaluation of trigonometric series  we refer the reader to \cite{gen69}.


\end{document}